\documentclass[12 pt]{amsart}
\usepackage{amsmath,amssymb,amsthm}
\usepackage{longtable}
\usepackage{a4wide}
\newcommand{\Sp}{\mathrm{Sp}}

\def \syl {\hbox {\rm Syl}}

\newtheorem{defn}{Definition}[section]
\newtheorem{theorem}[defn]{Theorem}

\newtheorem{lemma}[defn]{Lemma}

\newtheorem{prop}[defn]{Proposition}
\def \d {\mathrm d}
\title{The commuting graph of a soluble group}
\author{Christopher Parker}

\address{
School of Mathematics\\
University of Birmingham\\
Edgbaston\\
Birmingham B15 2TT\\
United Kingdom} \email{c.w.parker@bham.ac.uk}

\begin{document}
\maketitle
\begin{abstract} The commuting graph of a finite soluble group with trivial centre is investigated. It is shown that the diameter of such a graph is at most 8 or the graph is disconnected. Examples of soluble groups with trivial centre and commuting graph of diameter 8 are provided. 

\end{abstract}
\section{Introduction}

Suppose that $G$ is a    group. The \emph{commuting graph}  $\Gamma(G)$ of $G$ is the graph which has vertices the non-central elements of $G$ and two distinct vertices of $\Gamma(G)$ are adjacent if and only if they commute in $G$.

\begin{theorem}\label{1}
Suppose that $G$ is a finite soluble group with trivial centre. Then  
\begin{enumerate}\item[(i)] $\Gamma(G)$ is disconnected if and only if   $G$ is  a Frobenius group or a $2$-Frobenius group.
\item [(ii)] If $\Gamma(G)$ is connected, then $\Gamma(G)$ has diameter at most $8$. \end{enumerate}Furthermore, there exist soluble groups $G$ with trivial centre such that $\Gamma(G)$  has diameter $8$.
\end{theorem}

In Theorem~\ref{1} (i), a \emph{$2$-Frobenius group} is a group $G$ which has proper normal subgroups $K$ and $L$ such that $L$ is a Frobenius group with kernel $K$ and $G/K$ is a Frobenius group with kernel $L/K$. Note that if $G$ is a $2$-Frobenius group, then $G/K$ is metacyclic.

In \cite{G}, an infinite family of finite  soluble groups $G$ with $Z(G)=1$ and $\Gamma(G)$ connected of diameter $6$ is presented. These examples prompted the work described in this paper. In  his PhD thesis,  Woodcock \cite{W}  showed that if    $G$ is a non-trivial soluble group with trivial centre in which all the non-trivial Sylow $r$-subgroups, $r$ a prime, are not cyclic or, in addition, if $r=2$, not generalized quaternion, then $G$ has diameter at most $7$. In work of Segev and Seitz which was required for the study of the multiplicative group of a division algebra (they prove that finite quotients of the multiplicative group of a finite dimensional
division algebra are soluble in collaboration with Rapinchuk \cite{R}), it is demonstrated that the commuting graph of a classical simple group defined over a field of order greater than $5$ is either disconnected or has  diameter at most $10$ and at least $4$ \cite[Corollary (pg. 127), Theorem 8]{SS}.
They also prove that the commuting graphs of the exceptional Lie type groups other than $\mathrm E_7(q)$ and the sporadic simple groups are disconnected \cite[Theorem 6]{SS}. In \cite{I}  Iranmanesh and  Jafarzadeh demonstrate that the commuting graph of  $\mathrm{Sym}(n)$ and $\mathrm {Alt}(n)$ is either disconnected or has diameter  at most $5$ and, in the same article, they conjecture that there is an absolute upper bound for the diameter of a connected commuting graph of a non-abelian finite group. Theorem~\ref{1} confirms this conjecture for finite soluble groups with trivial centre. In \cite{Heg}, an example of a $2$-group with commuting graph of diameter 10 is given and they present a family of randomly created  $2$-groups of class two which they conjecture have unbounded diameter commuting graph.  None-the-less the conjecture that the diameter of the commuting graph of a group with trivial centre has an absolute upper bound may be correct.

In Section~3, we prove that  the commuting graph of a soluble group with trivial centre is either disconnected of has diameter at most $8$. Also in Section~3 and at the same time as we prove the main statement, we show that the commuting graph is disconnected if and only if $G$ is either a Frobenius or $2$-Frobenius group. This latter statement is most probably well-known as in this case the prime graph is also disconnected \cite{Gr, Will}.  Guided by our attempt to prove that the diameter of such a graph was at most $7$, in Section 4, we conjure up a family of connected commuting graphs of diameter $8$. One of the smallest groups in the family has order $11^{20}5^23221=54173193341944394740910525$.

If  $\{x,y\}$ is an edge in $\Gamma(G)$ or, if $ x=y$, then we write $x \sim y$. In particular, $x \sim y$ indicates that $x,y \in G \setminus Z(G)$. If $x $ and $ y $ are vertices in $\Gamma(G)$, then $\mathrm d(x,y)$ denotes the distance between $x$ and $y$. Our group theoretic notation is mostly standard and follows that in \cite{Aschbacher, Gor}. In particular we mention that for a group $G$, $G^\#$ is the set of non-identity elements of $G$.

\section{Preliminary results}

Before we engage in the proof of Theorem~\ref{1}, we post five preliminary results which are well-known and will be used frequently.

\begin{lemma}\label{Frob} Suppose that $X$ is a group and $X= JK$ with $J$ a proper normal subgroup of $X$ and $K$ a complement to $J$. Then  $C_X(k) \le K$ for all $k \in K^\#$ if and only if  $X$ is a Frobenius group.
\end{lemma}

\begin{proof} If $X$ is a Frobenius group then this statement is found in  \cite[2.7.6 (iv)]{Gor}. Suppose  that  $X= JK$ with $J$ a proper normal subgroup of $X$ and $K$ a complement to $J$. Assume that $h \in (K \cap K^g)^\#$ for some $g\in X$. Then, as $X= JK=KJ$, $g = kj$ for some $j \in J$ and $k \in K$. Thus $$K \cap K^g= K\cap K^j.$$ Since $h, h^{j^{-1}} \in K$ and $J$ is normal in $X$,  $h^{-1}h^{j^{-1}}= [h,j^{-1}] \in J\cap K=1$. Consequently $j \in C_X(h)  \le K\cap J=1$. It follows that if $K \cap K^g \ne 1$, then $g \in K$ and $K= K^g$. Hence $X$ is a Frobenius group.
\end{proof}

The next lemma is similar to the preceding one except and the proof is left as an exercise.

\begin{lemma}\label{Frob2} Suppose that $J$ is a proper normal subgroup of $X$. Then $X$ is a Frobenius group if and only if $C_X(j) \le J$ for all $j \in J^\#$.\qed
\end{lemma}

 The next lemma describes the structure of a Frobenius complement and goes back to Burnside.

\begin{lemma}\label{meta} Suppose $X$ is a Frobenius complement. Then every Sylow subgroup of $X$  is cyclic or
generalized quaternion. Furthermore, any two elements of prime order  commute and if $X$ has odd order then $X$ is metacyclic.
\end{lemma}
\begin{proof}
\cite[ Satz 8.18, p. 506]{Hu}.
\end{proof}

The next lemma delivers elements with non-trivial centralizers.

\begin{lemma}\label{fps} Let $p,q,r$ be distinct primes, $X$ be a group of order $r$ which acts faithfully on a $q$-group $Q$ and $V$ be a faithful $\mathrm{GF}(p)XQ$-module. Additionally, if $q=2$ and $r$ is a Fermat prime, assume that $Q$ is abelian. Then $C_V(X)\ne 0$.
\end{lemma}

\begin{proof}See \cite[36.2]{Aschbacher}.\end{proof}

One final celebrated result: Frobenius kernels are nilpotent.

\begin{theorem}[Thompson]\label{T}If $G$ admits a fixed-point-free automorphism of prime order, then $G$ is nilpotent.
\end{theorem}

\begin{proof} See
\cite[Theorem 10.2.1]{Gor} for example.
\end{proof}

\section{Proof of the theorem}
The first two lemmas confirm that the commuting graphs of Frobenius groups and $2$-Frobenius groups are disconnected. They are surely well-known.

\begin{lemma}\label{L0} If $X$ is a Frobenius group, then $\Gamma(X)$ is disconnected.
\end{lemma}

\begin{proof} Suppose  $K$ is a  Frobenius complement of $X$.  Then $C_X(k) \le K$ for all $k \in K^\#$ by Lemma~\ref{Frob}. Hence the vertices of $\Gamma$ in  $K^\#$ are only connected to vertices in $K^\#$ and this means $\Gamma(X)$ is disconnected with one of the  connected component having vertices contained in $K^\#$.
\end{proof}

\begin{lemma}\label{L01} If $X$ is a $2$-Frobenius group, then $\Gamma(X)$ is disconnected.
\end{lemma}

\begin{proof} Let $K$ and $L$ be normal subgroups of $X$ such that $L$ is a Frobenius group with kernel $K$ and $X/K$ is a Frobenius group with kernel $L/K$.
Let $J$ be a complement to $K$ in $L$ and $M = N_X(J)$. Then, by the Frattini Argument, $X= MK$  and $M \cap K= 1$. Hence $M$ is a complement to $K$ in $X$. Thus $M \cong X/K$ is a Frobenius group. We consider the subgraph of $\Gamma(X)$ spanned by the elements of $J^\#$ and claim that this is disconnected from $X\setminus J$. So let $j \in J^\#$ and consider $C_X(j)$. We have $C_M(j) \le J$ and $C_L(j) \le J$  as $L$ and $M$ are Frobenius groups. Let $x \in C_X(j)$. Then $x= mk$ for some $m \in M$ and $k \in K^\#$. Thus $j^x= j$ and this means  $j^m = j^{k^{-1}}$. Hence $j^m \in M$ and $j^{k^{-1}} \in L$. Therefore $$j^m =j^{k^{-1}}\in L \cap M=  JK \cap M= (K \cap M)J= J. $$
Since $M$ is a Frobenius group with complement $J$, $m \in J$ and, as $L$ is a Frobenius group with complement $J$, $k^{{-1}}\in J$. But then $x \in J$ and so  $C_X(j) \le J$. This proves our claim.
\end{proof}

The remainder of this section is devoted to the proof of Theorem~\ref{1}. Suppose that $G$ is a finite soluble group with $Z(G)=1$. Set $\Gamma= \Gamma(G)$. Because $G$ has trivial centre, $G^\#= G\setminus\{1\}$ is the vertex set of $\Gamma(G)$. Note that the next two lemmas apply to arbitrary finite groups.

\begin{lemma}\label{L1}  Assume that $N$ is a minimal normal subgroup of $G$ and set $F= C_G(N)$.  If  $a,b \in G^\#$ and $\d(a,b) > 4$, then either $C_G(a) \cap F = 1$ or $C_G(b) \cap F=1$.
\end{lemma}

\begin{proof}  If $C_G(a) \cap F \ne  1$ and $C_G(b) \cap F\ne 1$. Then there exist $n \in N^\#$ and $c, d \in F^\#$ such that $$a\sim c \sim n \sim d\sim b,$$ which is a contradiction as $d(a,b) > 4$.
\end{proof}

The next lemma more-or-less explains why the structure of   groups with a large diameter commuting graph have very uncomplicated and restricted structure.

\begin{lemma}\label{L2} Suppose that the diameter of $\Gamma$ is at least $7$. If $N$ is a minimal normal subgroup of $G$, then $C_G(N) = F(G)$.

\end{lemma}

\begin{proof} Set $F= C_G(N)$ and let $c,d \in G^\#$ be such that $\d(c,d)>6$. Since $F(G)$ is nilpotent, $F(G) \le F$. Let $a \in \langle c \rangle$ and $b \in \langle d \rangle$ have prime order.
Because $\d(c,d) > 6$, $\d(a,b) > 4$ and so  Lemma~\ref{L1} gives $C_F(a)=C_G(a) \cap F = 1$ or $C_F(b)=C_G(b) \cap F=1$.  Therefore Thompson's Theorem~\ref{T} shows $F$ is nilpotent. Since $F$ is normal in $G$, $F \le F(G)$ and so $F=F(G)$, as claimed.
\end{proof}

 From now on we fix  $F= F(G)$.  Let $J$ be the preimage of $F(G/F)$,  $Z$ be the preimage of $Z(J/F)$. We also let $V$ be a minimal normal subgroup of $G$. Thus $V \le F$ and  $F= C_G(V)$ by Lemma~\ref{L2}.

\begin{lemma}\label{L6}  Suppose that the diameter of $\Gamma$ is at least $7$. Then  $|J/F|$ is coprime to $|F|$.
\end{lemma}

\begin{proof} Assume that $s$ is a prime which divides $|F|$ and let $S \in \syl_s(J)$. Then $Z(SF) \ne 1$ and, as $SF$ is normal in $G$,  there exists a minimal normal subgroup $N$ of $G$ contained in $Z(SF)$. However, by Lemma~\ref{L2}, $F= C_G(N)$  thus $S  \le F$ and this proves the lemma.
\end{proof}

 For the remainder of this section, aiming to prove Theorem~\ref{1} by contradiction, we assume that \begin{quote}\emph{$G$ is not a Frobenius group or a $2$-Frobenius group and the diameter of $\Gamma$ is greater than $8$.  }\end{quote}

 Let $x, y \in G^\#$ be such that $\d(x,y)>8$. Note that there may be no path between $x$ and $y$ as we have not assumed that $\Gamma$ is connected.
 Let $x_r$ be an element of prime order $r$ in $\langle x\rangle$ and $y_s $ be an element of prime order $s$ in $\langle y\rangle$. Furthermore,  suppose that $r$ and $s$ are chosen to be maximal. Obviously $\d(x_r,y_s)> 6$.

\begin{lemma}\label{L4}  Either all elements of $C_G(x_r)^\#$ or all elements of $C_G(y_s)^\#$ act fixed-point-freely on $F$.
\end{lemma}

\begin{proof} Assume the statement is false and let $f \in C_F(a)^\#$ and $g \in C_F(b)^\#$ where $a \in C_G(x_r)^\#$ and $b \in C_G(y_s)^\#$. Then there exists $v \in V^\#$ such that $$x\sim x_r \sim a\sim f \sim v  \sim g \sim b \sim y_s \sim y,$$ which contradicts $\d(x,y)>8$.
\end{proof}

Because of Lemma~\ref{L4} we may, without loss of generality, assume that every element of $C_G(x_r)^\#$ acts fixed-point-freely on $F$.  Lemma~\ref{Frob} gives the following   immediate consequence of this choice.

\begin{lemma}\label{L5}
We have $C_G(x_r)F$ is a Frobenius group.  In particular, $G \ne C_G(x_r)F$. \qed
\end{lemma}

The next two lemmas could easily be combined into one for our forthcoming arguments; however, the separation of them more transparently indicates  the structure of  groups which are likely to provide extremal examples of large diameter commuting graphs and explains why in our examples constructed in Section~4 we have required $C_G(x_r)$ to have odd order.

\begin{lemma}\label{Cxrodd} The subgroup $C_G(x_r)$ has odd order and  is metacyclic.
\end{lemma}

\begin{proof} Suppose that $z \in C_G(x_r)$ is an involution. Then, as every non-trivial element of $C_G(x_r)$ acts fixed-point-freely on $F$, $z$ is the unique involution of $C_G(x_r)$ by Lemma~\ref{meta}.
 Therefore $F$ is inverted by $z$, $x$ centralizes $z$, $F$ is abelian and $G= C_G(z)F$.

As $G$ is not a Frobenius group, there exists $d \in C_G(z)^\#$ such that $C_F(d) \ne 1$ by   Lemma~\ref{Frob}. Assume $C_F(y_s) \ne 1$. Then we have $$x \sim z \sim d \sim f \sim f_1\sim y_s \sim y$$ where $f, f_1 \in F^\#$, a contradiction.
Hence $C_F(y_s)=1$ and $y$ centralizes   $z^f$ for some  $f \in F$. Observe $d^f \in C_G(z^f)$ and $C_F(d)= C_F(d)^f=C_F(d^f)$. Hence there exists $f_1 \in C_F(d)^\#$ such that $$x \sim z \sim d \sim f_1 \sim d^f \sim z^f \sim y$$ the final straw.    Hence $|C_G(x_r)|$ is odd and  the fact that $C_G(x_r)$ is metacyclic follows from Lemma~\ref{meta}. \end{proof}

The fact that $G$ is soluble plays a critical role in the next lemma.

\begin{lemma}\label{L7} We have $x_r \in Z$ and $Z/F$ is cyclic.
\end{lemma}

\begin{proof} Suppose that $x_r \not \in Z$. Then, as $J/F= F(G/F)$ and $G/F$ is soluble, there exists a prime $t$ such that the Sylow $t$-subgroup of $J/F$ is not centralized by $x_r$. Suppose first that $t= r$. Let $R \in \syl_r(J\langle x_r\rangle)$ with $x_r \in R$. Then $x_r\not\in Z(R)$. Let $z_r \in Z(R)^\#$ have order $r$. Then $\langle z_r, x_r \rangle$ is elementary abelian of order $r^2$ and, as $r$ is coprime to $|F|$, by \cite[Theorem 5.3.16]{Gor},  $$V= \langle C_V(z)\mid z \in \langle z_r, x_r \rangle^\#\rangle$$ contrary to the choice of $x_r$ to have every element of $C_G(x_r)^\#$ acting fixed-point-freely on $F$. Therefore $t \ne r$. Choose  $ T \in \syl_t(J)$ to be $\langle x_r\rangle$ invariant. Since $C_V(x_r)=0$, noting that $T\langle x_r\rangle$ acts faithfully on $V$ by Lemma~\ref{L2}, Lemma~\ref{fps}  implies that $T$ is a $2$-group and that every abelian characteristic subgroup of $T$ is centralized by $x_r$. In particular, $Z(T)$ is centralized by $x_r$. Thus $C_G(x_r)$ has even order which is contrary to Lemma~\ref{Cxrodd}. Thus $x_r$ centralizes $J/F$ and so $x_r \in Z$. Finally as some complement to $Z$ in $F$ is contained in $C_G(x_r)$ and all abelian subgroups of $C_G(x_r)$ are cyclic by Lemma~\ref{meta}, $Z/F$ is cyclic.
\end{proof}

\begin{lemma}\label{L11} We have  $G= N_G(\langle x_r\rangle)F$, $N_G(\langle x_r\rangle) \cap F=1$ and $G/C_G(x_r)F$ is cyclic of order dividing $r-1$.
\end{lemma}

\begin{proof} Since, by Lemma~\ref{L7}, $Z/F$ is cyclic, $\langle x_r\rangle F$ is a normal subgroup of $G$. Hence $G= N_G(\langle x_r\rangle)F$ by the Frattini  Argument. As $F \cap C_G( x_r) =1$ and $F$ is normal in $G$,   $N_G(\langle x_r\rangle) \cap F=1$. Because $x_r$ has order $r$ and $\mathrm{Aut}(\langle x_r\rangle)$ is cyclic of order $r-1$, we obtain $G/C_G(x_r)F$ is cyclic of order dividing $r-1$.
\end{proof}

\begin{lemma}\label{L13}
There exists $d \in G^\#$, $c \in C_G(x_r)^\#$ and $v \in V^\#$  such that $$x_r\sim c \sim d \sim v.$$   Furthermore, $c$ can be chosen of prime order.\end{lemma}

\begin{proof} Recall from Lemma~\ref{L5}, $C_G(x_r)F$ is a Frobenius group  and  $G \ne C_G(x_r)F$. Let $K$ be chosen so that $ C_G(x_r) \le K \le N_G( \langle x_r\rangle)$ is maximal by containment  such that $KF$ is a Frobenius group. As $G$ is not a Frobenius group, $G \ne KF$. Observe that, since $G/C_G(x_r)F$ is cyclic by Lemma~\ref{L11}, $KF$ is normal in $G$ and $K$ is   normal in   $N_G( \langle x_r\rangle)$.

If $C_G(k) \le K$ for all $k \in K^\#$, then, in particular, $C_{N_G(\langle x_r \rangle)}(k) \le K$ for all $k \in K^\#$. Since $K$ is a proper normal subgroup of $N_G(\langle x_r \rangle)$, $N_G(\langle x_r \rangle)$ is a Frobenius group with kernel $K$ by Lemma~\ref{Frob2}. But then Lemma~\ref{L13}   and the choice of $K$ implies that $G$ is a $2$-Frobenius group, which is a contradiction.  Hence there exists $k \in K^\#$ such that $C_G(k) \not \le K$. Let $k\in K$ be any such  element.   If $g \in C_G(k)$, then $K \cap K^g \ge \langle k \rangle$ and so, as $KF$ is a Frobenius group, $g$ normalizes $K$. Thus $C_G(k) \le N_G(K) = N_G(\langle x_r\rangle)$ and this is true for all $k \in K^\#$.

Let $d_0 \in N_G( \langle x_r\rangle)\setminus K$ be of minimal order  such that $C_K(d_0) \ne 1$. Then $|K\langle d_0\rangle: K| =t$ for some prime $t$. From the maximal choice of $K$, $K \langle d_0 \rangle $ is not a Frobenius complement in $F K\langle d_0 \rangle$ and so   there exists $d \in K \langle d_0 \rangle$ such that $C_F(d)\ne 1$ by Lemma~\ref{Frob}.  Since $C_F(k)= 1$ for all $k \in K^\#$, we have $\langle d \rangle \cap K = 1$.  Therefore, from the minimal choice of $d_0$, we have $K \langle d_0 \rangle= K\langle d \rangle$ and $d$ has order $t$.  Now,  either $d$ centralizes some $t$-element of $K$ or $t$ and $|K|$ are coprime. In the latter case, as  $C_K(d_0^n) \ge C_K(d_0)$ for all integers $n$, the minimal choice of the order of $d_0$ implies that $d_0$ has order $t$ as well. Therefore, in this case,   $\langle d \rangle$ and $\langle d_0\rangle $ are conjugate in $K\langle d_0 \rangle$ by Sylow's Theorem. Thus whatever happens we have $C_K(d) \ne 1$. If $V$ is a $t$-group then $C_V(d)\ne 1$.  So assume that $V$ is not a $t$-group. Then,
as $d \in N_G( \langle x_r\rangle)\setminus K$ and $K \ge C_G(x_r)$, $d$ acts non-trivially on $\langle x_r\rangle$ and so,  by Lemma~\ref{fps}, $C_V(d) \ne 1$.
Now let $c \in C_K(d)$ be chosen of prime order. Since $c \in K$ and $K$ is a Frobenius complement in $KF$, $C_V(c)\le C_F(c)=1$ and, since $c$ normalizes $\langle x_r \rangle$, $c \in C_G(x_r)$ by Lemma~\ref{fps} (or \ref{meta}).
Thus choosing  $v \in C_V(d)^\#$ we have the result  as claimed.
\end{proof}

\begin{lemma}\label{L14}  Suppose that $e \in C_G(y_s)^\#$. Then $C_F(e)=  1$.
\end{lemma}

\begin{proof} If this statement is false, then  $C_F(e)\ne 1$ for some $e \in C_G(y_s)^\#$. Let $f  \in C_F(e)^\#$.  Lemma~\ref{L13} provides  $c \in C_G(x_r)^\#$,  $d \in G^\#$,  $v \in  V^\# \subseteq Z(F)^\#$    such that 
$$x\sim x_r \sim c \sim d \sim v \sim f\sim e \sim y_s \sim y$$ 
which contradicts $\d(x,y) > 8$.
\end{proof}

\begin{proof}[Proof of Theorem~\ref{1}]
Using Lemma~\ref{L13}  there exists $c \in C_G(x_r)^\#$, $d \in G^\#$ and  $v  \in V^\#$ such that
$x_r \sim c \sim d \sim v$.   Moreover we may suppose that $c$ has prime order.  Since $|C_G(x_r)|$ has odd order and any two elements of prime order commute in $C_G(x_r)$, we have $c \in Z(F(C_G(x_r)))$ and so $c \in Z$ and $\langle x_r,c \rangle$ are contained in a complement to $F$ in $Z$. Because of Lemma~\ref{L14} the situation between $x_r$ and $y_s$ is symmetric. In particular, $y_s \in Z$ by Lemma~\ref{L7}. Therefore, there exists an element $h$ of $F$ such that $\langle y_s, c^h \rangle$ is contained in a complement to $F$ in $Z$. Especially, $c^h$ centralizes $y_s$. Hence, as $V \le Z(F)$, we have the following path between $x$ and $y$:
$$x\sim x_r  \sim c\sim d \sim v=v^h  \sim d^h\sim c^h \sim  y_s \sim y$$ which contradicts $\d(x,y) > 8$.  This contradiction finally disproves the hypothesis that $G$ in not a Frobenius group  or a $2$-Frobenius group and the diameter of $\Gamma$ is greater than $8$. Hence  $G$ is either a Frobenius group or a $2$-Frobenius group  or $\Gamma$ has diameter at most $8$. In the former cases, Lemmas~\ref{L0} and \ref{L01} say $\Gamma$ disconnected. This proves Theorem~\ref{1}. \end{proof}

\section{Examples of diameter 8}

In this section we present a series of examples of  soluble groups $G$ with trivial centre such that $\Gamma(G)$ is connected and the diameter of $\Gamma(G)$ is $8$. This demonstrates that the bound on the diameter of $\Gamma(G)$ obtained in Theorem~\ref{1} (ii) is optimal.
We adopt the notation from the previous section. The examples appear in a situation similar to that studied in Lemma~\ref{L4}.

First we select $q$ a power of an odd prime  such that the prime $r$, with $r \ge 5$, divides $q-1$ exactly. Let  $\mathbb F = \mathrm{GF}(q^r)$ and  $\beta$ be the Frobenius automorphism of $\mathbb F$ of order $r$. Note that $r^2$ divides $q^r-1$. Let $t$ be a prime which divides $(q^r-1)/(q-1)$ but not $q-1$. As an example of this type of formation we may take $q=11$. In this case $r=5$ and $t= 3221$.

  Let $V$ be a $4$-dimensional symplectic space over $\mathbb F$ defined with respect to the form which has  matrix
   $J=\left(\begin{smallmatrix} 0&0&0&1\\
0&0&1&0\\
0&-1&0&0\\
-1&0&0&0\end{smallmatrix}\right)$. Then $H= \Sp(V)= \{A \in \mathrm{GL}(V)\mid AJA^T=J\}$ where $A^T$ is the transpose of $A$. Let

$$F = \left\{\left(\begin{smallmatrix}
1&0&0&0\\
a&1&0&0\\
b&x&1&0\\
c&d&-a&1
\end{smallmatrix}\right)\mid a,b,c,d,x \in \mathbb F, xa=d-b\right\}.$$
Then $F$ has order $q^{4r}$ and, since $q$ is odd, $F$ has nilpotency class $3$. In fact $F$ is a Sylow subgroup of $H$, but this is unimportant for our considerations.
Let $$z= \left(\begin{smallmatrix}
d&0&0&0\\
0&e&0&0\\
0&0&e^{-1}&0\\
0&0&0&d^{-1}\end{smallmatrix}\right)$$ have order $r^2$ and be such that $d, e \in \mathbb F$ are elements of order $r^2$ and the set $\{d^r,d^{-r},e^r,e^{-r}\}$ has order $4$ (which is why we need $r \ge 5$).
 Then $z \in N_H(F)$ and it is elementary to check that $C_F(z^r)=1$ as $x$ acts on $V$ with distinct eigenvalues. Let $$c=  \left(\begin{smallmatrix}
f&0&0&0\\
0&f&0&0\\
0&0&f^{-1}&0\\
0&0&0&f^{-1}\end{smallmatrix}\right)$$ where $f \in \mathbb F$ has order $t$. This time we calculate $$C_F(c) = \left\{\left(\begin{smallmatrix}
1&0&0&0\\
a&1&0&0\\
0&0&1&0\\
0&0&-a&1
\end{smallmatrix}\right)\mid a\in \mathbb F\right\}$$ which therefore has order $q$ and is elementary abelian.

Recall that $\beta$ induces an automorphism of $H$ by applying $\beta$ to the entries of the matrices in $H$. We also denote this automorphism by $\beta$.  Notice that $\beta$ normalizes $F$, $\langle z\rangle$ and $\langle c\rangle$.  Consider $x=z\beta $ in the semi-direct product $H\rtimes \langle \beta\rangle$. Let $D = \langle x, c\rangle$.

\begin{lemma}\label{Dstruct} $x^r$ has order $r$ and commutes with $c$ and $c^x= c^q \ne c$. In particular, $D$ is metacyclic of order $r^2t$  and $Z(D)=\langle x^r\rangle$.
\end{lemma}

\begin{proof} We just have to consider the elements $d$ and $e$ of $\mathbb F$. We have  $(d \beta)^ r= d^{q^{r-1}+ \dots + q +1} \in \mathrm{GF}(q)$ (the fixed field of $\beta$ when acting on $\mathbb F$)and thus has order $r$. A similar statement holds for $e$ and therefore $x^r$ is a diagonal matrix and commutes with $c$.

Since $z$ commutes with $c$, we  have $c^x= c^\beta= c^q$.

Finally, as $\langle c \rangle$ is normalized by $x$, we have $D/\langle c \rangle \cong \langle x \rangle$. This proves the result because $\langle x \rangle$ has order $r^2$ and $\langle c \rangle $ has order $t$ is not centralized by $x$ but is centralized by $x^r$.
\end{proof}

Set $G= FD \le H \rtimes \langle \beta\rangle$. This is the group which we will demonstrate has a connected commuting graph of diameter $8$. Proving this statement is now our main aim. We let $\Gamma= \Gamma(G)$.

\begin{lemma}
$G$ is not a Frobenius or $2$-Frobenius group. In particular, $\Gamma$ is connected.
\end{lemma}

\begin{proof}  Suppose the statement is false. Obviously $G$ is not a Frobenius group as $C_F(c) \ne 1$. Thus there are proper normal subgroups $K$ and $L$ of $G$ such that $L$ is a Frobenius group with kernel $K$ and $G/K$ is a Frobenius group with kernel $L/K$. Since the kernels of Frobenius groups coincide with the Fitting subgroup, we have $K \le F(G)=F$ and $F/K \le L/K$. But then $F \le L$ and so $K= F$. Since, by Lemma~\ref{Dstruct}, $G/F$ is not a Frobenius group we have a contradiction. Now Theorem~\ref{1} (i) says that $\Gamma$ is connected.
\end{proof}

\begin{lemma} The following statements hold.
\begin{enumerate}
\item [(i)]$C_G(x) = \langle x\rangle$;
\item [(ii)]$C_G(x_r) = D$ for any element $x_r$ of order $r$ in $C_G(x)$; and
\item [(iii)]$C_G(c^*) = \langle c,x^r\rangle C_F(c)$ for any element $c^*$ of order $t$ in $C_X(x^r)$.
\end{enumerate}
\end{lemma}

\begin{proof} We have already seen
$C_F(x)=1$ and $C_D(x)=\langle x \rangle$. Hence (i) holds. We also know $C_F(x^r)= 1$, so as $c$ and $x^r$ commute, we have (ii). Finally, if $w \in D^\#$ and $C_F(w) \ne 1$, then $w \in \langle c\rangle$ and this means (iii) holds.
\end{proof}

We now set $g= \left(\begin{smallmatrix}
1&0&0&0\\
0&1&0&0\\
1&1&1&0\\
0&1&0&1
\end{smallmatrix}\right) \in F$  and put $y= x^g$.

\begin{lemma}\label{M3} If  $g\in C_F(c)^\#$  and $h \in (C_F(c)^g)^\#$, then $[g,h]\ne 1$.
\end{lemma}

\begin{proof} We have $$C_F(c)^g = \left\{\left(\begin{smallmatrix}
1&0&0&0\\
a&1&0&0\\
-a&0&1&0\\
0&-a&-a&1
\end{smallmatrix}\right)\mid a\in \mathbb F\right\}$$ and then readily calculate that the lemma holds.
\end{proof}

\begin{prop}\label{M4} $\d(x,y) = 8$ and $\Gamma$ has diameter $8$.
\end{prop}

\begin{proof} Notice that, if $a \sim b \sim c$, then, for all integers $n$, $a \sim b^n \sim c$ so long as $b^n \ne 1$. In particular, this means that we may  assume that in a path from $x$ to $y$, the inner terms of the path have prime order.
Since $C_G(x)= \langle x \rangle$ and $C_G(y)= \langle y\rangle$, any path from $x$ to $y$ must have the form $$x \sim x^r\sim  \dots \sim (x^r)^g \sim x^g= y.$$ By Lemma~\ref{M3} (ii), $C_G(x^r)= D$ and the only way to advance from $x$ is to move to an element of order $t$. There is a unique cyclic group of this order in $D$ and so we must have
 $$x \sim x^r\sim c \sim  \dots \sim c^g \sim (x^r)^g \sim x^g= y.$$
 The next vertex must be contained in  $$C_G(c)= \langle x^r \rangle \langle c \rangle C_F(c)$$ by Lemma~\ref{M3}(iii).
The elements of prime order in $C_G(c)$ are conjugate to elements of $\langle x^r\rangle$, $C_F(c)$ or $\langle c \rangle$.  Suppose that $x \sim x^r \sim c \sim (x^r)^h$ for some $h \in C_G(c)$. Then the ``prime"  neighbours of $(x^r)^h$ are in the abelian group $\langle (x^r)^h, c \rangle$ and so there is no way we can take this route without returning via some power of $c$. It follows that there must exists $w \in C_F(c)^\#$ and $w^* \in C_F(c^g)^\#$ such that
 $$x \sim x^r\sim c \sim w \sim \dots \sim w^* \sim c^g \sim (x^r)^g \sim x^g= y.$$
By Lemma~\ref{M3}, there are no choices of $w\in C_F(c)^\#$ and $w^* \in C_F(c^g)^\#$ such that $[w,w^*]=1$. Since $F$ is nilpotent, we have $\d(w,w^*)=2$ and so we conclude $\d(x,y)=8$.   As $\Gamma$ is connected, we obtain $\Gamma$ has diameter $8$ by using Theorem~\ref{1}.
 \end{proof}

Finally we note that as $q=11$, $r=5$ and $t= 3221$ satisfies our initial criteria we have an explicit example of order $11^{20}5^23221$.

\end{document}